\documentclass[12pt,a4paper,reqno]{amsart}

\usepackage[top=35mm, bottom=35mm, left=30mm, right=30mm]{geometry}
\usepackage[T1]{fontenc}
\usepackage{amsmath,amssymb,amsthm,mathtools}
\usepackage{enumitem,verbatim}
\usepackage{microtype}

\usepackage[hidelinks,pagebackref]{hyperref}
\renewcommand*{\backref}[1]{}
\renewcommand*{\backrefalt}[4]{%
\ifcase #1
Not cited%
\or
(Cited on page~#2)%
\else
(Cited on pages~#2)%
\fi
}

\usepackage[looser]{newtxtext}
\usepackage[smallerops]{newtxmath}

\newtheorem{theorem}{Theorem}[section]
\newtheorem{corollary}[theorem]{Corollary}
\newtheorem{proposition}[theorem]{Proposition}
\newtheorem{lemma}[theorem]{Lemma}
\newtheorem{question}[theorem]{Question}
\theoremstyle{definition}

\newtheorem{remark}[theorem]{Remark}

\newcommand{\N}{\mathbb N}

\newcommand{\dens}{\overline{\operatorname{dens}}}
\newcommand{\card}{\operatorname{card}}
\newcommand{\Rad}{\mathcal R}
\newcommand{\Env}{\mathcal E}

\title[Power growth of mean-L-stable operators]{\large Power Growth of mean-L-stable Operators on Banach Spaces}

\author[Jian Li]{Jian Li}
\address[Jian Li]{Institute for  Mathematical Sciences and Artificial Intelligence \& Department of Mathematics,
	Shantou University, Shantou, 515821, Guangdong, China}
\email{lijian09@mail.ustc.edu.cn}
\urladdr{https://orcid.org/0000-0002-8724-3050}

\author[Jie Li]{Jie Li$^*$}  
\address[Jie Li]{School of Mathematics and Statistics, Jiangsu Normal University, Xuzhou, Jiangsu, 221116, P.R. China}
\email{jiel0516@mail.ustc.edu.cn}
\urladdr{https://orcid.org/0000-0001-7147-4746}

\thanks{$^{*}$ Corresponding author.}

\date{\today}



\begin{document}

\begin{abstract}
We study the growth of powers of mean-L-stable operators on Banach spaces. We show that, for linear operators, mean-L-stability is equivalent to uniform boundedness in density; this yields $\|T^n\|=O(n)$ on every Banach space. On Hilbert spaces we prove that mean-L-stability is equivalent to absolute Ces\`aro boundedness and obtain $\|T^n\|=O(n^{1/2-c_T})$ for some $c_T>0$. For positive mean-L-stable operators on abstract $L^p$-spaces, $1\le p<\infty$, we similarly obtain $\|T^n\|=O(n^{1/p-c_T})$, where in both cases the positive constant $c_T$ cannot be chosen uniformly over all such operators.

For positive mean-L-stable operators on $p$-convex Banach lattices, we prove the bound $O(n^{1/p})$ and construct positive topologically mixing operators $T_p$ for which $\|T_p^n\|\asymp n^{1/p}$, where $1\leq p<\infty$. 
These operators satisfy a uniform weak $(p,p)$ orbit estimate, while the averages of $\|T_p^nx\|^s$ are bounded for $s<p$, of order $\log N$ for $s=p$, and of order $N^{s/p-1}$ for $s>p$. 
The operator $T_1$ is uniformly Kreiss bounded and has linear power growth, answering a question of Montes-Rodr\'iguez, S\'anchez-\'Alvarez and Zem\'anek (2005). 
Moreover, $T_1$ is mean-L-stable and mean Li--Yorke chaotic, while it is not distributionally chaotic. This answers a question of Bernardes, Bonilla and Peris (2020).
\end{abstract}

\maketitle

\section{Introduction}
The growth of powers of operators under Ces\`aro and resolvent conditions has been studied extensively. Sharp estimates under Kreiss-type conditions were obtained in \cite{EFR02,BM21}, while absolute Ces\`aro boundedness and related power-growth problems were studied in \cite{BBMP20,CCEL20,C20}. 
In particular, Cuny obtained estimates depending on the type and cotype of the underlying space, with applications to $L^p$-spaces \cite{C20}. 
More recent power-growth estimates for strongly Kreiss bounded operators on $L^p$-spaces and related classes of Banach spaces can be found in \cite{DLV24,AC26}.
These works show that both the boundedness assumption and the geometry of the underlying space influence the possible growth of $\|T^n\|$. 
We investigate the corresponding problem for mean-L-stable operators.

The concept of stability in the mean in the sense of Lyapunov, abbreviated mean-L-stability, was introduced by Fomin in connection with topological dynamical systems with pure point spectrum \cite{F51}.
It requires small perturbations to remain small except on sets of times of arbitrarily small upper density. 
On compact metric systems, Li, Tu and Ye showed that mean-L-stability is equivalent to mean equicontinuity \cite[Lemma~3.1]{LTY15}.  
For operators on a Banach space, Jiang and Li showed that absolute Ces\`aro boundedness is equivalent to mean equicontinuity \cite[Theorem~4.32]{JL25}.
By Markov's inequality, absolute Ces\`aro boundedness implies mean-L-stability. 
Both the failure of mean-L-stability and absolute Ces\`aro boundedness have proved useful in the study of distributional chaos \cite{BBMP13,JL25} and mean Li--Yorke chaos \cite{BBP20} for linear operators, as well as in the study of typical orbit behavior of hypercyclic operators \cite{L24,LWZ26}.

The boundedness of the metric is important in these comparisons.  
On a compact metric space, the density formulation of mean-L-stability and the Ces\`aro formulation of mean equicontinuity are equivalent \cite{LTY15}.  
On a Banach space the norm is unbounded.  
A set of exceptional times may have small density while the orbit is arbitrarily large at those times, so density control need not imply a Ces\`aro bound.  
We therefore retain Fomin's density condition for the norm metric.  Linearity reduces the two-point condition to the orbit of a single vector, while a Baire category argument makes the resulting estimate uniform both in the vector and over finite time intervals.  
Throughout the paper we adopt the following convention: all operators are assumed to be bounded and linear unless otherwise stated.
For an operator $T$ on a Banach space $X$ and $R>0$, put
\[
 \Phi_T(R)=\sup_{\|x\|\le 1}\sup_{N\ge 1}
 \frac1N\card\{1\le j\le N:\|T^jx\|>R\}.
\]
We prove that
\[
 T\text{ is mean-L-stable}\quad\Longleftrightarrow\quad
 \Phi_T(R)\longrightarrow0 \quad\text{as }R\to\infty.
\]
We call this property \emph{uniform boundedness in density}; it is the form of mean-L-stability used throughout the paper.

This formulation can be compared directly with the usual boundedness conditions for operators. 
An operator $T$ on a Banach space $X$ is said to be \emph{power bounded} if $\sup_{n\ge 1}\|T^n\|<\infty$.
By the Banach--Steinhaus theorem, an operator $T$ is power bounded if and only if $\sup_{n\ge 1}\|T^nx\|<\infty$ for all $x\in X$.
For $1\le p<\infty$, an operator $T$ on a Banach space $X$ is called \emph{$p$-absolutely Ces\`aro bounded} if there
exists $C > 0$ such that
\[
\sup_{n\ge 1} \frac{1}{n}\sum_{k=0}^{n-1} \|T^kx\|^p\le C \|x\|^p,\quad \forall x\in X.
\]
The case $p=1$ was introduced by Luo and Hou in \cite{LH15} under the name absolute Ces\`aro boundedness.
The case $p=2$ was studied by van Casteren \cite{V83,V97} under the name square bounded in average; it was later referred to as mean square boundedness in \cite{BBMP20} and as Ces\`aro square boundedness in \cite{CCEL20}.

By H\"older's inequality, every $p$-absolutely Ces\`aro bounded operator is $r$-absolutely Ces\`aro bounded for every $1\le r\le p$. 
For absolutely Ces\`aro bounded operators, Berm\'udez, Bonilla, M\"uller and Peris showed that $\|T^n\|=o(n)$ on every Banach space and $\|T^n\|=o(n^{1/2})$ on Hilbert spaces \cite[Corollary~2.6 and Theorem~2.4]{BBMP20}.  
Cohen, Cuny, Eisner and Lin proved that, for a $p$-absolutely Ces\`aro bounded operator $T$, one has $\|T^n\|=O(n^{1/p-\varepsilon_T})$ for some $\varepsilon_T>0$ \cite[Proposition~3.1]{CCEL20}. They also proved that absolute Ces\`aro boundedness is equivalent to Ces\`aro square boundedness for Hilbert space operators \cite[Theorem~4.4]{CCEL20}.

The notion of uniform Kreiss boundedness was introduced in \cite{MSZ05}. We use the following equivalent formulation; see \cite[Corollary~3.2]{MSZ05}.
An operator $T$ on a complex Banach space is called 
\emph{uniformly Kreiss bounded} if 
there exists a constant $C>0$ such that 
\[
\sup_{n\ge 1} \biggl\| \frac{1}{n}\sum_{k=0}^{n-1}(\lambda T)^k\biggr\|\le C, \quad \forall |\lambda|=1.
\]
It is immediate that absolute Ces\`aro boundedness implies uniform Kreiss boundedness. Moreover, every uniformly Kreiss bounded operator satisfies $\|T^n\|=O(n)$ on an arbitrary Banach space.
Montes-Rodr\'iguez, S\'anchez-\'Alvarez and Zem\'anek asked the following question in \cite[Question~3]{MSZ05}: 
\begin{question} \label{ques:UKB-O-n}
 For operators on Banach spaces, does uniform Kreiss boundedness imply a better growth estimate of the powers than the general bound $O(n)$?
\end{question}

In \cite{AS16}, Aleman and Suciu further asked whether $\|T^n\|=o(n)$ whenever $T$ is a uniformly Kreiss bounded operator on a Banach space. Berm\'udez, Bonilla, M\"uller and Peris restated this question in \cite[Question~1.1]{BBMP20} and proved in \cite[Theorem~2.3]{BBMP20} that every uniformly Kreiss bounded operator on a Hilbert space satisfies $\|T^n\|=o(n)$.
Cuny proved in \cite[Theorem~3.1]{C20} that every (uniformly) Kreiss bounded operator on a Banach space with the unconditional martingale differences property satisfies $\|T^n\|=o(n)$.  
On the other hand, uniformly Kreiss bounded operators on Hilbert spaces may have powers of order $n^{1-\varepsilon}$ for any fixed $\varepsilon>0$ \cite[Theorem~4]{BM21}.   

We first place mean-L-stability among the standard boundedness conditions for operators. 
Power boundedness implies absolute Ces\`aro boundedness, which in turn implies uniform Kreiss boundedness \cite{BBMP20}, and neither converse holds in general. 
Markov's inequality also shows that absolute Ces\`aro boundedness implies mean-L-stability, while Theorem~\ref{thm:critical-family-main}, through the operator $T_1$, shows that the converse fails. 
Moreover, the examples in \cite[Theorem~4]{BM21}, together with our Hilbert space growth estimate, show that uniform Kreiss boundedness does not imply mean-L-stability even on Hilbert spaces. 
Under additional geometric or order assumptions, however, some of these conditions become equivalent.

\begin{theorem}\label{thm:relationships}
\textup{(i)} For a positive operator on an AM-space, power boundedness, mean-L-stability and absolute Ces\`aro boundedness are equivalent.

\smallskip 
\textup{(ii)} For an operator on a Hilbert space or a positive operator on an AL-space, mean-L-stability is equivalent to absolute Ces\`aro boundedness.

\smallskip 
\textup{(iii)} For a positive operator on a complex Banach lattice, mean-L-stability implies uniform Kreiss boundedness.
\end{theorem}

Our second main result establishes power-growth bounds under mean-L-stability.
\begin{theorem}\label{thm:main-growth}
Let $T$ be a mean-L-stable operator on a Banach space $X$.
\begin{enumerate}[label=\textup{(\roman*)}]
\item If $X$ has Rademacher type $p$, then
\[
 \|T^n\|=O(n^{1/p}).
\]
In particular, $\|T^n\|=O(n)$ on every Banach space.
\item If $X$ is a Hilbert space, then there is $c_T\in(0,1/2)$ such that
\[
 \|T^n\|=O(n^{1/2-c_T}).
\]
\item If $X$ is a $p$-convex Banach lattice for $1\le p<\infty$ and $T$ is positive, then
\[
 \|T^n\|=O(n^{1/p}).
\]
If, in addition, $X$ is an abstract $L^p$-space, then there is $c_T\in(0,1/p)$ such that
\[
 \|T^n\|=O(n^{1/p-c_T}).
\]
\end{enumerate}
\end{theorem}

For comparison, Cuny \cite{C20} obtained power-growth estimates for absolutely Ces\`aro bounded operators in terms of the type and cotype of the underlying Banach space. In particular, on a Banach space of Rademacher type $p$, absolute Ces\`aro boundedness implies $\|T^n\|=O(n^{1/p})$. Theorem~\ref{thm:main-growth}\textup{(i)} shows that this estimate persists under the weaker assumption of mean-L-stability. Indeed, absolute Ces\`aro boundedness implies mean-L-stability, while the converse fails in general, as shown below. Our proof uses only uniform control of the density of large orbit values, rather than uniform bounds on Ces\`aro averages of orbit norms.

The positive constants $c_T$ in Theorem~\ref{thm:main-growth}\textup{(ii)} and in the abstract $L^p$ case of Theorem~\ref{thm:main-growth}\textup{(iii)} cannot be chosen uniformly. Indeed, for every $0<\gamma<1/p$, the classical weighted shifts on $\ell^p$ recalled in Section~\ref{sec:critical} are absolutely Ces\`aro bounded and satisfy $\|T^n\|\asymp n^\gamma$. Since $\gamma$ can be arbitrarily close to $1/p$, there is no $c>0$ such that $\|T^n\|=O(n^{1/p-c})$ holds for all such operators. Taking $p=2$ gives the same conclusion for Hilbert spaces.
 
 \medskip
The order $n^{1/p}$ in the $p$-convex part of Theorem~\ref{thm:main-growth}\textup{(iii)} is sharp, as shown by the following family of examples.  

\begin{theorem}\label{thm:critical-family-main}
For $1\le p<\infty$, let
\[
 I_j=\{2^j,\ldots,2^{j+1}-1\}\text{ and }X_p=\biggl(\bigoplus_{j\ge 0}\ell^p(I_j;\mathbb C)\biggr)_{c_0}.
\]
Let $T_p$ be the unilateral weighted backward shift on $X_p$ with weights $\bigl(\frac{m+1}{m}\bigr)^{1/p}$, that is 
\[
 (T_px)_m=\left(\frac{m+1}{m}\right)^{1/p}x_{m+1}.
\]
Then $X_p$ and $T_p$ have the following properties:
\begin{enumerate}[label=\textup{(\roman*)}]
    \item The space $X_p$ is $p$-convex with constant one, and $T_p$ is a positive topologically mixing operator satisfying
    \[
     (n+1)^{1/p}\le\|T_p^n\|\le2^{1/p}(n+1)^{1/p},\quad \forall n\ge  1.
    \]
    \item There is an absolute constant $C_0>0$ such that
\[
 \frac1N\card\{1\le n\le N:\|T_p^nx\|>\lambda\}
 \le C_0\left(\frac{\|x\|}{\lambda}\right)^p
\]
for all $x\in X_p$, $\lambda>0$ and $N\ge 1$.
In particular, each $T_p$ is  mean-L-stable.
\item For every $s>0$, 
\[
 \sup_{\|x\|\le 1}\frac1N\sum_{n=1}^N\|T_p^nx\|^s
 \asymp_{p,s}
 \begin{cases}
 1,&0<s<p,\\
 \log(N+1),&s=p,\\
 N^{s/p-1},&s>p.
 \end{cases}
\]
In particular, each $T_p$ is $s$-absolutely Ces\`aro bounded if and only if $1\le  s<p$.
\end{enumerate}
\end{theorem}

Thus the general $p$-convex estimate has the optimal order $n^{1/p}$. For the same operators, the averages of $\|T_p^n x\|^s$ are uniformly bounded when $s<p$, have order $\log N$ when $s=p$, and have order $N^{s/p-1}$ when $s>p$. In particular, $T_1$ is mean-L-stable and has linear power growth, while the Ces\`aro averages of its orbit norms are not uniformly bounded.

\begin{remark}
The operator $T_1$ in Theorem~\ref{thm:critical-family-main} is a positive mean-L-stable operator on a complex Banach lattice. By Theorem~\ref{thm:relationships}\textup{(iii)}, $T_1$ is uniformly Kreiss bounded. However,
\[
n+1\le  \|T_1^n\|\le  2(n+1) \quad(n\ge 1).
\]
Thus the general $O(n)$ power-growth bound for uniformly Kreiss bounded operators cannot be improved.
\end{remark}

Consequently, we obtain a negative answer to Question~3 of \cite{MSZ05} and to the later more specific question of Aleman and Suciu \cite{AS16}; see also Question~\ref{ques:UKB-O-n}.

\begin{proposition}
There exists a uniformly Kreiss bounded operator $T$ on a Banach space $X$ with $\|T^n\| \asymp n$.
\end{proposition}

We conclude the introduction with an application to linear dynamics. Distributional chaos and mean Li--Yorke chaos are two important notions of chaos for linear operators; see \cite{LY16} for related background. It was shown in \cite[Theorem~25]{BBPW18} that there exists an operator on $\ell^p$ which is distributionally chaotic but not mean Li--Yorke chaotic. Bernardes, Bonilla and Peris asked in \cite[Question~16]{BBP20} whether the opposite separation can occur:

\begin{question}\label{ques:mean-chaos}
Is there a Banach (or Hilbert) space operator which is mean Li--Yorke chaotic but is not distributionally chaotic?
\end{question}

The operator $T_1$ in Theorem~\ref{thm:critical-family-main} is mean-L-stable but not absolutely Ces\`aro bounded. Proposition~\ref{prop:T1-chaos} shows that $T_1$ is mean Li--Yorke chaotic, while every pair is distributionally asymptotic. 
In particular, $T_1$ is not distributionally chaotic, and hence Question~\ref{ques:mean-chaos} has an affirmative answer on Banach spaces. 
In contrast, using Theorem~\ref{thm:relationships}, we prove in Corollary~\ref{cor:Hilbert-MLY-DC} that every mean Li--Yorke chaotic operator on a Hilbert space is distributionally chaotic. 
Thus no such example exists on a Hilbert space.

\medskip
The paper is organized as follows. In Section~\ref{sec:prelim}, we recall the Banach-space and Banach-lattice notions used below and establish convenient uniform characterizations of mean-L-stability. In Section~\ref{sec:relations}, we compare mean-L-stability with power boundedness, absolute Ces\`aro boundedness and uniform Kreiss boundedness under Hilbert-space and lattice assumptions. In Section~\ref{sec:power-growth}, we prove the general power-growth estimates and their refinements on Hilbert spaces and positive $p$-convex Banach lattices. In Section~\ref{sec:critical}, we construct the operators $T_p$ and establish their sharp power growth, weak orbit estimates and Ces\`aro behavior. Finally, in Section~\ref{sec:chaos}, we study the relation between mean Li--Yorke chaos and distributional chaos and apply the preceding results to the questions discussed above.

\section{Preliminaries}\label{sec:prelim}

\subsection{Banach spaces and Banach lattices} 
In this subsection, we recall the Banach-space and Banach-lattice notions used below. 
We refer to \cite{AK16}, \cite{K85} and \cite{M91} for standard references.

For a Banach space $X$ and $n\in\mathbb N$, define
\[
\Rad_X(n)=\sup_{x_1,\ldots,x_n\in B_X}\mathbb E\biggl\|\sum_{j=1}^n\varepsilon_jx_j\biggr\|,
\]
where $B_X=\{x\in X:\|x\|\le 1\}$ and $\varepsilon_1,\ldots,\varepsilon_n$ are independent Rademacher random variables, each taking the values $-1$ and $1$ with probability $1/2$. We shall also refer to the $\varepsilon_j$ as Rademacher signs.
If a Banach space $X$ has Rademacher type $p\in [1,2]$, then there exists a constant $C>0$ such that
$\Rad_X(n)\le Cn^{1/p}$ for all $n\ge  1$.

For a Banach lattice $X$, let $X_+=\{x\in X:x\ge 0\}$ denote its positive cone. For $n\in\mathbb N$, define
\[
 \Env_X(n)=\sup_{x_1,\ldots,x_n\in B_X\cap X_+}\|x_1\vee\cdots\vee x_n\|.
\]
We use the standard lattice functional calculus throughout the paper.  For $1\le p<\infty$,  
a Banach lattice $X$ is called \emph{$p$-convex} if there is $M_p(X)\ge 1$ such that
\[
 \biggl\|\biggl(\sum_{j=1}^m|x_j|^p\biggr)^{1/p}\biggr\|
 \le M_p(X)\biggl(\sum_{j=1}^m\|x_j\|^p\biggr)^{1/p}
\]
for every finite family $x_1,\dotsc,x_m\in X$.  
Standard examples of $p$-convex Banach lattices are $L^p(\mu)$ and $\ell^p$.
An \emph{AM-space} is a Banach lattice satisfying
\[
 \|x\vee y\|=\max\{\|x\|,\|y\|\},
 \quad \forall x,y\in X_+.
\]
Typical examples of AM-spaces are spaces with a supremum-type norm, such as $C(K)$ and $c_0$.
It follows directly that every AM-space is $p$-convex for every $1\le p<\infty$ with $p$-convexity constant $1$.
A Banach lattice $X$ is called an \emph{abstract $L^p$-space} if, for every $x,y\in X_+$ with $x\wedge y=0$,
\[ \|x+y\|^p = \|x\|^p +\|y\|^p.  
\]
A Banach lattice $X$ is called an \emph{AL-space} if the norm is additive on the positive cone, that is,
\[
    x,y\in X_+\implies \|x+y\|=\|x\|+\|y\|.
\]
The spaces $L^1(\mu)$ and $\ell^1$ are standard examples of AL-spaces.

For a $p$-convex Banach lattice $X$ and $m\ge 1$, put
\[
 \mathfrak C_{p,X}(m)
 =\sup_{\substack{1\le \ell\le m\\ y_1,\ldots,y_\ell\in X_+\setminus\{0\}}}
 \frac{(\sum_{i=1}^\ell\|y_i\|^p)^{1/p}}
 {\|(\sum_{i=1}^\ell y_i^p)^{1/p}\|}.
\]
The immediate bounds are
\[
 1\le\mathfrak C_{p,X}(m)\le m^{1/p}.
\]
A Banach lattice $X$ is \emph{$p$-concave} if there is $C_p(X)\ge 1$ such that
\[
 \biggl(\sum_{i=1}^m\|y_i\|^p\biggr)^{1/p}
 \le C_p(X)\biggl\|\biggl(\sum_{i=1}^m|y_i|^p\biggr)^{1/p}\biggr\|
\]
for every finite family $y_1,\dotsc,y_m\in X$.  
Hence $\sup_m\mathfrak C_{p,X}(m)\le C_p(X)$.  
An abstract $L^p$-space is both $p$-convex and $p$-concave with constant one, so $\mathfrak C_{p,X}(m)=1$; see \cite{M91}.

We shall use the following elementary numerical lemma to establish the preliminary power-growth estimates below.

\begin{lemma}\label{lem:reciprocal-sequence}
Let $(a_n)_{n\ge 1}$ be a positive sequence. Suppose that for some $B\ge 1$,
\begin{equation}\label{eq:reciprocal-growth}
a_N\sum_{r=1}^{\lfloor N/8\rfloor}\frac{1}{a_r}\leq BN
\end{equation}
for all sufficiently large $N$. Then there exist $C\ge 1$ and $\delta>0$ such that
\[
a_n\leq Cn^{1-\delta}
\]
for every $n\ge 1$.
\end{lemma}

\begin{proof}
Choose $N_0\ge 1$ such that \eqref{eq:reciprocal-growth} holds for every $N\geq N_0$, and choose $k_0\ge 0$ such that $2^{k_0}\geq N_0$. For $k\geq k_0$, set $b_k=\max_{2^k\leq n<2^{k+1}}\frac{a_n}{n}$, $c_k=b_k^{-1}$, and choose $N_k\in[2^k,2^{k+1})$ such that $a_{N_k}/N_k=b_k$.

Fix $k\geq k_0+5$. If $k_0\leq i\leq k-5$ and $2^i\leq r<2^{i+1}$, then
\[
r<2^{i+1}\leq 2^{k-4}\leq \frac{N_k}{16}<\frac{N_k}{8}.
\]
Hence the whole dyadic block $[2^i,2^{i+1})$ occurs in the sum in \eqref{eq:reciprocal-growth} with $N=N_k$. Moreover, $a_r/r\leq b_i$, and hence 
\[
\sum_{2^i\leq r<2^{i+1}}\frac{1}{a_r}
\geq c_i\sum_{2^i\leq r<2^{i+1}}\frac{1}{r}
\geq \frac{c_i}{2}.
\]
Applying \eqref{eq:reciprocal-growth} with $N=N_k$, we obtain
\[
B\geq \frac{a_{N_k}}{N_k}\sum_{r=1}^{\lfloor N_k/8\rfloor}\frac{1}{a_r}
\geq \frac{b_k}{2}\sum_{i=k_0}^{k-5}c_i.
\]
Thus, for every $k\geq k_0+5$, we have 
\begin{equation}\label{eq:ck-recursion}
c_k\geq \frac{1}{2B}\sum_{i=k_0}^{k-5}c_i.
\end{equation}

Set $S_k=\sum_{i=k_0}^k c_i$ and $\lambda=\frac{1}{2B}$. By \eqref{eq:ck-recursion}, we have 
\[
S_k=S_{k-1}+c_k\geq (1+\lambda)S_{k-5}
\]
for every $k\geq k_0+5$. Put $\rho=(1+\lambda)^{1/5}>1$. Iterating the last inequality separately on the five residue classes modulo $5$, we obtain $S_k\geq C_0\rho^k$ for every $k\geq k_0$, where
$C_0=\min_{0\leq j\leq 4}S_{k_0+j}\rho^{-(k_0+j)}>0$. 
Consequently, \eqref{eq:ck-recursion} implies
\[
c_k\geq \lambda C_0\rho^{k-5}=C_1\rho^k
\]
for every $k\geq k_0+5$, where $C_1=\lambda C_0\rho^{-5}>0$. Hence
\[
b_k\leq C_2\rho^{-k},
\]
where $C_2=C_1^{-1}$.

Let $\delta=\log_2\rho=\frac{1}{5}\log_2\left(1+\frac{1}{2B}\right)>0$. Then $b_k\leq C_2 2^{-\delta k}$ for every $k\geq k_0+5$. If $2^k\leq n<2^{k+1}$, then
\[
a_n\leq nb_k\leq C_2n2^{-\delta k}\leq 2^\delta C_2n^{1-\delta}.
\]
Increasing the constant to cover the finitely many remaining values of $n$ completes the proof.
\end{proof}

\subsection{Mean-L-stability}
Let $T$ be an operator on a Banach space $X$.
We say that $T$ is \emph{mean-L-stable} if, for every $\varepsilon>0$, there exists $\delta>0$ such that
\[
 \|x-y\|<\delta
 \quad\Longrightarrow\quad
 \dens\{n\in\N:\|T^nx-T^ny\|\ge\varepsilon\}<\varepsilon,
\]
where, for $A\subset\N$,
\[
 \dens(A)=\limsup_{N\to\infty}\frac1N\card(A\cap\{1,\ldots,N\}).
\]
By linearity this is equivalent to the one-point formulation with $\| x\|$ and $\|T^nx\|$.
 
In analogy with the Banach–Steinhaus theorem, a Baire category argument gives the following uniform finite-time formulation, see e.g.\@ \cite[Corollary~5.16]{JL25}.

\begin{proposition}\label{prop:uniform-density}
An operator $T$ on a Banach space $X$ is mean-L-stable if and only if, for every $\varepsilon>0$, there exists $\delta>0$ such that
\[
 \sup_{N\in\N}\frac1N
 \card\{1\le j\le N:\|T^jx\|\ge\varepsilon\}<\varepsilon
\]
for every $x\in X$ with $\|x\|<\delta$.
\end{proposition}

We shall also use the following consequence of \cite[Theorem~5.13, Proposition~5.33 and Theorem~5.41]{JL25}.

\begin{proposition}\label{prop:MLS-dichotomy}
Let $T$ be an operator on a Banach space $X$.
Then either $T$ is mean-L-stable or there exist $x\in X$ and $A\subset\N$ with $\dens(A)=1$ such that
\[
 \lim_{A\ni n\to\infty}\|T^nx\|=\infty.
\]
\end{proposition}

By homogeneity, mean-L-stability is equivalent to the following statement: for every $0<\eta<1$ there is $K>0$ such that
\[
 \dens\{n\in\N:\|T^nx\|>K\|x\|\}<\eta,\quad 
 \forall x\ne 0.
\]
One direction follows by scaling a unit vector into the $\delta$-ball.  
For the converse, take a density level smaller than the prescribed tolerance and scale back.  
We will prove that an estimate of the following form is equivalent to mean-L-stability.

For $0<\eta<1$, put
\begin{align*}
\mathcal K_T(\eta)=\inf \biggl\{K\ge 1\colon
&\exists N_0\in\N\text{ s.t. }\forall x\ne 0,\ \forall N\ge N_0, \\
&\frac1N\card\bigl\{1\le j\le N:\|T^jx\|>K\|x\|\bigr\}\le\eta\biggr\}.
\end{align*}
The value $+\infty$ is allowed.
If $T$ is absolutely Ces\`aro bounded with constant $C$, then Markov's inequality, applied to a Ces\`aro average of length $N+1$, shows that
\[
 \mathcal K_T(\eta)\le\max\left\{1,\frac{2C}{\eta}\right\}.
\]

\begin{proposition}\label{prop:density-characterization}
Let $T$ be an operator on a Banach space $X$.
Then the following assertions are equivalent:
\begin{enumerate}[label=\textup{(\roman*)}]
    \item $T$ is mean-L-stable;
    \item $\mathcal K_T(\eta)<\infty$ for every $0<\eta<1$;
    \item $\mathcal K_T(\eta_0)<\infty$ for some $0<\eta_0<1$.
\end{enumerate}
\end{proposition}

\begin{proof}
Assume first that $T$ is mean-L-stable, and fix $0<\eta<1$. 
By Proposition~\ref{prop:uniform-density}, applied with $\eta/2$, there exists $\delta>0$ such that
\[
 \sup_{N\in\N}\frac1N
 \card\{1\le j\le N:\|T^jz\|\ge\eta/2\}<\eta/2
\]
whenever $\|z\|<\delta$. Put $K=\max\left\{1,\frac{\eta}{\delta}\right\}$. For $x\ne 0$, set $z=\delta x/(2\|x\|)$. Then $\|z\|<\delta$, and
\[
 \|T^jx\|>K\|x\|
 \quad\Longrightarrow\quad
 \|T^jz\|>\frac{\delta K}{2}\ge\frac{\eta}{2}.
\]
It follows that, for every $N\in\N$,
\[
 \frac1N\card\bigl\{1\le j\le N:\|T^jx\|>K\|x\|\bigr\}
 <\frac{\eta}{2}<\eta.
\]
Thus $\mathcal K_T(\eta)\le K<\infty$. Since $\eta$ was arbitrary, $\mathcal K_T(\eta)<\infty$ for every $0<\eta<1$.
Thus \textup{(i)} implies \textup{(ii)}. The implication \textup{(ii)}$\Rightarrow$\textup{(iii)} is immediate.

It remains to prove \textup{(iii)}$\Rightarrow$\textup{(i)}. Suppose that $\mathcal K_T(\eta_0)<\infty$ for some $0<\eta_0<1$. Choose $K\ge 1$ and $N_0\ge 1$ such that for every $x\neq 0$ and every $N\ge N_0$, 
\[
 \frac1N\card\bigl\{1\le j\le N:\|T^jx\|>K\|x\|\bigr\}\le\eta_0.
\]
Suppose that $T$ is not mean-L-stable. By Proposition~\ref{prop:MLS-dichotomy}, there exist $x\in X$ and $A\subset\N$ with $\dens(A)=1$ such that
\[
 \lim_{A\ni n\to\infty}\|T^nx\|=\infty.
\]
Clearly $x\ne 0$. Hence there exists $n_0\ge 1$ such that
\[
 n\in A,\ n\ge n_0
 \quad\Longrightarrow\quad
 \|T^nx\|>K\|x\|.
\]
Removing finitely many integers does not change upper density, and therefore
\[
 \dens\{n\in\N:\|T^nx\|>K\|x\|\}=1.
\]
On the other hand, the choice of $K$ and $N_0$ gives
\[
 \dens\{n\in\N:\|T^nx\|>K\|x\|\}\le\eta_0<1,
\]
a contradiction. Thus $T$ is mean-L-stable.
\end{proof}

We use the term \emph{uniform boundedness in density} for the following uniform estimate. 
The next proposition shows that it is equivalent to mean-$L$-stability.
For $R>0$, define the finite-time tail function for orbits of the unit ball by
\[
 \Phi_T(R)=\sup_{\|x\|\le 1}\sup_{N\ge 1}\frac1N
 \card\{1\leq j\leq N:\|T^jx\|>R\}.
\] 

\begin{proposition}\label{prop:uniform-tail}
Let $T$ be an operator on a Banach space $X$.
Then the following assertions are equivalent:
\begin{enumerate}[label=\textup{(\roman*)}]
\item $T$ is mean-L-stable;
\item $\Phi_T(R)\to0$ as $R\to\infty$;
\item there exists $R>0$ such that $\Phi_T(R)<1$.
\end{enumerate}
\end{proposition}

\begin{proof}
Assume first that $T$ is mean-L-stable, and fix $0<\eta<1$. By Proposition~\ref{prop:uniform-density}, applied with $\eta/2$, there exists $\delta>0$ such that, for every $z\in X$ with $\|z\|<\delta$,
\[
 \sup_{N\in\N}\frac1N
 \card\{1\le j\le N:\|T^jz\|\ge\eta/2\}<\eta/2.
\]
Put $R=\eta/\delta$. If $\|x\|\le1$ and $z=\delta x/2$, then $\|z\|<\delta$ and
\[
 \|T^jx\|>R
 \quad\Longrightarrow\quad
 \|T^jz\|>\eta/2.
\]
Hence
\[
 \sup_{N\in\N}\frac1N
 \card\{1\le j\le N:\|T^jx\|>R\}<\eta/2.
\]
Taking the supremum over $\|x\|\le1$ gives $\Phi_T(R)\le\eta/2$. Since $\Phi_T$ is decreasing and $\eta>0$ is arbitrary, it follows that
\[
 \Phi_T(R)\longrightarrow0
 \quad\text{as }R\to\infty.
\]
Thus \textup{(i)} implies \textup{(ii)}. The implication \textup{(ii)}$\Rightarrow$\textup{(iii)} is immediate.

It remains to prove \textup{(iii)}$\Rightarrow$\textup{(i)}. 
Suppose that $T$ is not mean-L-stable. By Proposition~\ref{prop:MLS-dichotomy}, there exist $x\in X$ and $A\subset\N$ with $\dens(A)=1$ such that
\[
 \lim_{A\ni n\to\infty}\|T^nx\|=\infty.
\]
Clearly $x\ne 0$. Put $u=x/\|x\|$. Fix $R>0$. There exists $n_0\ge 1$ such that
\[
 n\in A,\ n\ge n_0
 \quad\Longrightarrow\quad
 \|T^nu\|>R.
\]
Removing finitely many integers does not change upper density, and hence
\[
 \dens\{n\in\N:\|T^nu\|>R\}=1.
\]
Therefore
\[
 \sup_{N\ge 1}\frac1N
 \card\{1\le j\le N:\|T^ju\|>R\}=1.
\]
Since $\|u\|=1$, we obtain $\Phi_T(R)=1$. As $R>0$ was arbitrary, $\Phi_T(R)=1$ for every $R>0$, contrary to \textup{(iii)}. Thus $T$ is mean-L-stable.
\end{proof}

\section{Relations between mean-L-stability and absolute Ces\`aro boundedness}\label{sec:relations}

In this section, we investigate the relationships among  various versions of boundedness and we prove Theorem~\ref{thm:relationships}.
We start with positive operators on an AM-space.

\begin{proposition}\label{prop:AM-power-bounded}
Let $T$ be a positive operator on an AM-space $X$.
If $\Phi_T(K)<2^{-1}$ for some $K\ge 1$, then $ \|T^n\|\leq K^2$ for all $n\ge 1$. 
In particular, every positive mean-L-stable operator on an AM-space is power bounded.
\end{proposition}

\begin{proof}
Fix $n\ge 1$ and $x\in B_X\cap X_+$. Put
\[
 G=\{0\leq k<n:\|T^kx\|\leq K\} \text{ and }
 J=\{n-k:k\in G\}.
\]
If $n=1$, then $G=\{0\}$ and hence $\card(J)=1>n/2$. If $n\geq2$, the definition of
$\Phi_T(K)$, applied with $N=n-1$, gives
\[
 \card(J)=\card(G)>\frac n2.
\]

For $j=n-k\in J$, set $ u_j=K^{-1}T^kx$. Then $u_j\in B_X\cap X_+$ and $T^ju_j=K^{-1}T^nx$. Let $ y=\bigvee_{j\in J}u_j$. Since $X$ is an AM-space,
\[
 \|y\|=\max_{j\in J}\|u_j\|\le 1.
\]
Moreover, positivity of $T$ implies for each $j\in J$, 
\[
 T^jy\geq T^ju_j=K^{-1}T^nx.
\]

Suppose that $\|T^nx\|>K^2$. Then, for every $j\in J$,
\[
 \|T^jy\|\geq K^{-1}\|T^nx\|>K.
\]
Consequently,
\[
 \frac1n\card\{1\leq j\leq n:\|T^jy\|>K\}
 \geq\frac{\card(J)}n>\frac12,
\]
which contradicts $\Phi_T(K)<1/2$, since $\|y\|\le 1$. Thus $\|T^nx\|\leq K^2$ for every $x\in B_X\cap X_+$.

Finally, let $z\in B_X$ be arbitrary. Since $T$ is positive,
\[
 |T^nz|\leq T^n|z|.
\]
As $|z|\in B_X\cap X_+$, it follows that
\[
 \|T^nz\|\leq\|T^n|z|\|\leq K^2.
\]
Hence $\|T^n\|\leq K^2$ for all $n\ge 1$, and therefore $T$ is power bounded.
\end{proof}

Cohen, Cuny, Eisner and Lin proved in \cite[Theorem 4.4]{CCEL20} that every absolutely Ces\`aro
bounded operator on a Hilbert space is Ces\`aro
square bounded. 
We can strengthen this result with the following quantitative estimate.

\begin{proposition}\label{prop:Hilbert-square}
Let $H$ be a Hilbert space and let $T\colon H\to H$ be an operator. 
Suppose that $\Phi_T(R)<2^{-10}$ for some $R\ge 1$. Then, for every $x\in H$ and every $M\ge 1$, 
\begin{equation}\label{eq:Hilbert-block-square}
\frac1M\sum_{k=M}^{2M-1}\|T^kx\|^2\le 128R^4\|x\|^2.
\end{equation}
Consequently,
\begin{equation}\label{eq:Hilbert-square-global}
\sup_{N\ge 1}\frac1N\sum_{k=0}^{N-1}\|T^kx\|^2\le 257R^4\|x\|^2.
\end{equation}
\end{proposition}

\begin{proof}
By homogeneity it is enough to prove \eqref{eq:Hilbert-block-square} when $\|x\|=1$. Fix $M\ge 1$ and put
\[
L=\frac1M\sum_{k=M}^{2M-1}\|T^kx\|^2.
\]
If $L=0$, there is nothing to prove. Let $G=\{0\le  j<M:\|T^jx\|\le  R\}$ and let $g=\card(G)$. Since $R\ge 1$ and $\|x\|=1$, we have $0\in G$. If $M=1$, then $g=1>M/2$. If $M\ge 2$, the definition of $\Phi_T(R)$ with $N=M-1$ shows that fewer than $2^{-10}(M-1)$ indices $1\le  j\le  M-1$ satisfy $\|T^jx\|>R$. Hence
\[
g>M-2^{-10}(M-1)>\frac M2.
\]

Let $(\varepsilon_j)_{j\in G}$ be independent Rademacher signs and put
\[
h=\frac1{R\sqrt g}\sum_{j\in G}\varepsilon_jT^jx.
\]
By orthogonality of the Rademacher signs,   
\[
\mathbb E\|h\|^2=\frac1{R^2g}\sum_{j\in G}\|T^jx\|^2\le 1.
\]
Markov's inequality then yields $\mathbb P(\|h\|>4)\le 1/16$.

For $1\le  n\le 2M-1$, put
\[
q_n=\frac1{R^2g}\sum_{\substack{j\in G\\ M\le  n+j<2M}}\|T^{n+j}x\|^2.
\]
For each $j\in G$ and each $M\le  k<2M$, there is exactly one $n=k-j$ with $1\le  n\le 2M-1$. It follows that
\[
\sum_{n=1}^{2M-1}q_n=\frac1{R^2g}\sum_{j\in G}\sum_{k=M}^{2M-1}\|T^kx\|^2=\frac{ML}{R^2}.
\]
Since the sum in the definition of $q_n$ contains only terms from the block $M\le  k<2M$,  we have 
\[
q_n\le \frac{ML}{R^2g}<\frac{2L}{R^2}.
\]
Let
\[
Q=\left\{1\le  n\le 2M-1:q_n>\frac{L}{4R^2}\right\}.
\]
We claim that $\card(Q)>M/4$. Otherwise,
\[
\sum_{n=1}^{2M-1}q_n\le \frac M4\frac{2L}{R^2}+(2M-1)\frac{L}{4R^2}<\frac{ML}{R^2},
\]
which contradicts the preceding identity.

Fix $n\in Q$ and put
\[
a_n=\mathbb E\|T^nh\|^2=\frac1{R^2g}\sum_{j\in G}\|T^{n+j}x\|^2.
\]
The terms occurring in $q_n$ are among the terms in this sum. Hence
\[
a_n\ge  q_n>\frac{L}{4R^2}.
\]
We shall also use the following fourth-moment estimate for Hilbert-space-valued Rademacher sums:
\[
\mathbb E\biggl\|\sum_{i=1}^m\varepsilon_iv_i\biggr\|^4\le 3\biggl(\sum_{i=1}^m\|v_i\|^2\biggr)^2.
\]
For completeness, if $A=\sum_{i=1}^m\|v_i\|^2$, then independence of the Rademacher signs implies
\[
\mathbb E\biggl\|\sum_{i=1}^m\varepsilon_iv_i\biggr\|^4=A^2+4\sum_{i<j}\bigl(\operatorname{Re}\langle v_i,v_j\rangle\bigr)^2\le 3A^2.
\]
Applying this estimate to the vectors $(R\sqrt g)^{-1}T^{n+j}x$, $j\in G$, we obtain
\[
\mathbb E\|T^nh\|^4\le 3a_n^2.
\]
The Paley--Zygmund inequality with parameter $1/2$, applied to the nonnegative random variable $\|T^nh\|^2$, now implies
\[
\mathbb P\left(\|T^nh\|^2>\frac12a_n\right)\ge \frac14\frac{a_n^2}{\mathbb E\|T^nh\|^4}\ge \frac1{12}.
\]
Since $a_n>L/(4R^2)$, we conclude that for every $n\in Q$, 
\[
\mathbb P\left(\|T^nh\|>\frac{\sqrt L}{2\sqrt2R}\right)\ge \frac1{12}.
\]

Let $Z$ denote the number of $n\in Q$ such that $\|T^nh\|>\sqrt L/(2\sqrt2R)$, and let $A=\{\|h\|\le 4\}$. Then $\mathbb EZ\ge \card(Q)/12$. Since $0\le  Z\le \card(Q)$ and $\mathbb P(A^c)\le 1/16$, we have 
\[
\mathbb E(Z\mathbf1_A)\ge \mathbb EZ-\card(Q)\mathbb P(A^c)\ge \frac{\card(Q)}{12}-\frac{\card(Q)}{16}=\frac{\card(Q)}{48}.
\]
Thus there is a choice of the signs for which $\|h\|\le 4$ and
\[
Z\ge \frac{\card(Q)}{48}>\frac{M}{192}.
\]
Fix such a choice and put $z=h/4$. Then $\|z\|\le 1$, and for more than $M/192$ indices $1\le  n\le 2M-1$ we have
\[
\|T^nz\|>\frac{\sqrt L}{8\sqrt2R}.
\]
Suppose that $L>128R^4$. Then $\sqrt L/(8\sqrt2R)>R$, and therefore
\[
\frac1{2M-1}\card\{1\le  n\le 2M-1:\|T^nz\|>R\}>\frac{M}{192(2M-1)}>\frac1{384}>2^{-10}.
\]
This contradicts $\Phi_T(R)<2^{-10}$. Hence $L\le 128R^4$, which proves \eqref{eq:Hilbert-block-square}.

It remains to prove \eqref{eq:Hilbert-square-global}. Let $N\ge 2$ and choose $J$ such that $2^J\le  N-1<2^{J+1}$. Applying \eqref{eq:Hilbert-block-square} with $M=2^r$ for $0\le  r\le  J$, we obtain
\[
\sum_{k=1}^{N-1}\|T^kx\|^2\le 128R^4\sum_{r=0}^J2^r\|x\|^2<256R^4N\|x\|^2.
\]
After adding the term $k=0$ and using $R\ge 1$, we have 
\[
\frac1N\sum_{k=0}^{N-1}\|T^kx\|^2\le 257R^4\|x\|^2.
\]
The case $N=1$ is immediate.
\end{proof}

\begin{theorem}\label{thm:Hilbert-position}
Let $H$ be a real or complex Hilbert space and let $T\colon H\to H$ be an operator.  
Then the following are equivalent:
\begin{enumerate}[label=\textup{(\roman*)}]
\item $T$ is mean-L-stable;
\item $T$ is absolutely Ces\`aro bounded;
\item $T$ is Ces\`aro square bounded.
\end{enumerate}
If $H$ is complex, these conditions imply that $T$ is uniformly Kreiss bounded.
\end{theorem}

\begin{proof}
By Proposition~\ref{prop:uniform-tail}, mean-L-stability provides $R\ge 1$ such that $\Phi_T(R)<2^{-10}$.  Proposition~\ref{prop:Hilbert-square} then yields Ces\`aro square boundedness.  
The Cauchy--Schwarz inequality yields absolute Ces\`aro boundedness, while absolute Ces\`aro boundedness implies mean-L-stability by Markov's inequality.  
In the complex case absolute Ces\`aro boundedness implies uniform Kreiss boundedness by the definitions.
\end{proof}

We next compare mean-L-stability with absolute Ces\`aro boundedness and uniform Kreiss boundedness for positive operators on Banach lattices.

\begin{proposition}\label{prop:positive-lattice-UKB}
Let $T$ be a positive and mean-L-stable operator on a Banach lattice $X$. If $X$ is complex, then there is $C_T>0$ such that
\[
\sup_{N\ge 1}\sup_{|\gamma_0|=\cdots=|\gamma_{N-1}|=1}\biggl\|\frac1N\sum_{k=0}^{N-1}\gamma_kT^k\biggr\|\le C_T.
\]
In particular, $T$ is uniformly Kreiss bounded. 
If $X$ is an AL-space, then $T$ is absolutely Ces\`aro bounded.
\end{proposition}

\begin{proof}
By Proposition~\ref{prop:uniform-tail}, choose $R\ge 1$ such that $\Phi_T(R)<1/10$. 
Fix $x\in X_+$ and $x^*\in(X^*)_+$ with $\|x\|=\|x^*\|=1$, and write $a_k=\langle x^*,T^kx\rangle\ge 0$. 
We claim that
\begin{equation}\label{eq:positive-dyadic-block}
\frac1M\sum_{k=M}^{2M-1}a_k\le4R^2,\quad \forall \ M\ge 1.
\end{equation}
Suppose instead that $B=M^{-1}\sum_{k=M}^{2M-1}a_k>4R^2$ for some $M$. 
Let $G=\{0\le j<M:\|T^jx\|\le R\}$. If $M=1$, then $\card(G)=1$. If $M\ge2$, the definition of $\Phi_T(R)$ with $N=M-1$ implies
\[
\card(G)>M-\frac{M-1}{10}.
\]
Hence in all cases $\card(G)>9M/10$. 
Set
\[
h=\frac1{R\card(G)}\sum_{j\in G}T^jx.
\]
Then $h\ge 0$ and $\|h\|\le 1$. For $1\le n\le2M-1$, put
\[
b_n=\frac1{R\card(G)}\sum_{\substack{j\in G\\ M\le n+j<2M}}a_{n+j}.
\]
By positivity, $0\le b_n\le\langle x^*,T^nh\rangle\le\|T^nh\|$. 
Double counting yields $\sum_{n=1}^{2M-1}b_n=BM/R$, while $b_n<10B/(9R)$. Let $Q=\{1\le n\le2M-1:b_n>B/(4R)\}$. 
If $\card(Q)\le M/5$, then
\[
\frac{BM}{R}=\sum_{n=1}^{2M-1}b_n\le\frac M5\frac{10B}{9R}+2M\frac{B}{4R}<\frac{BM}{R},
\]
a contradiction. 
Hence $\card(Q)>M/5$. 
Since $B/(4R)>R$, more than $M/5$ indices $1\le n\le2M-1$ satisfy $\|T^nh\|>R$. 
Therefore
\[
\frac1{2M-1}\card\{1\le n\le2M-1:\|T^nh\|>R\}>\frac1{10},
\]
contrary to the choice of $R$. 
This proves \eqref{eq:positive-dyadic-block}.

A dyadic decomposition, together with $a_0\le 1$, now yields
\begin{equation}\label{eq:positive-scalar-Cesaro}
\sup_{N\ge 1}\frac1N\sum_{k=0}^{N-1}\langle x^*,T^kx\rangle\le8R^2
\end{equation}
for positive unit vectors $x$ and $x^*$. 
For $x\in X_+$, the positive vector $N^{-1}\sum_{k=0}^{N-1}T^kx$ has a positive norming functional, 
so \eqref{eq:positive-scalar-Cesaro} implies
\[
\biggl\|\frac1N\sum_{k=0}^{N-1}T^kx\biggr\|\le8R^2\|x\|.
\]
Suppose now that $X$ is complex. 
For arbitrary $x\in X$ and unimodular scalars $\gamma_0,\ldots,\gamma_{N-1}$, positivity yields
\[
\biggl|\sum_{k=0}^{N-1}\gamma_kT^kx\biggr|\le\sum_{k=0}^{N-1}|T^kx|\le\sum_{k=0}^{N-1}T^k|x|.
\]
Hence
\[
\biggl\|\frac1N\sum_{k=0}^{N-1}\gamma_kT^k\biggr\|\le8R^2.
\]
In particular, $T$ is uniformly Kreiss bounded.

If $X$ is an AL-space, then for $x\ge 0$ the norm is additive on the positive cone, and hence
\[
\frac1N\sum_{k=0}^{N-1}\|T^kx\|=\biggl\|\frac1N\sum_{k=0}^{N-1}T^kx\biggr\|\le8R^2\|x\|.
\]
For arbitrary $x$, positivity gives $|T^kx|\le T^k|x|$, and therefore the same estimate follows by applying the preceding inequality to $|x|$. 
Thus $T$ is absolutely Ces\`aro bounded.
\end{proof}

We can now prove Theorem~\ref{thm:relationships}.

\begin{proof}[Proof of Theorem~\ref{thm:relationships}]
Part \textup{(i)} follows from Proposition~\ref{prop:AM-power-bounded}, since a positive mean-L-stable operator on an AM-space is power bounded. 
Power boundedness implies absolute Ces\`aro boundedness, and absolute Ces\`aro boundedness implies mean-L-stability by Markov's inequality.

For \textup{(ii)}, the Hilbert-space case follows from Theorem~\ref{thm:Hilbert-position}. For a positive operator on an AL-space, Proposition~\ref{prop:positive-lattice-UKB} shows that mean-L-stability implies absolute Ces\`aro boundedness, while the converse follows from Markov's inequality.

Part \textup{(iii)} follows from Proposition~\ref{prop:positive-lattice-UKB}.
\end{proof}

\section{Power growth of mean-L-stable operators}
\label{sec:power-growth}

In this section, we prove Theorem~\ref{thm:main-growth}. We begin with a common estimate that converts uniform tail control into bounds determined by the geometry of the underlying space.

\begin{theorem}\label{thm:uniform-common}
Let $T$ be an operator on a Banach space $X$.
\begin{enumerate}[label=\textup{(\roman*)}]
\item If $\Phi_T(K)<2^{-3}$ for some $K\ge 1$, then  
\[
 \|T^n\|\leq 4K^2 \Rad_X(n), \quad \forall \ n\ge 1.
\]
In particular, if $X$ has Rademacher type $p$ and $T$ is mean-L-stable, then $\|T^n\|=O(n^{1/p})$.
\item Assume that $X$ is a Banach lattice and $T$ is a positive operator.
If $\Phi_T(K)<2^{-1}$ for some $K\ge 1$, then 
\[
 \|T^n\|\leq K^2\Env_X(n) \quad \forall \ n\ge 1.
\]
If, in addition, $X$ is $p$-convex, this implies $\|T^n\|=O(n^{1/p})$.
\end{enumerate}
\end{theorem}

\begin{proof}
\textup{(i)} Fix $n\ge 1$ and $x\in B_X$. 
If $T^nx=0$, there is nothing to prove. 
Put
\[
 G=\{0\leq k<n:\|T^kx\|\leq K\}.
\]
If $n=1$, then $G=\{0\}$ and hence $\card(G)=1>7n/8$. If $n\geq2$, the definition of $\Phi_T(K)$ with $N=n-1$ yields
\[
 \card(G)>n-\frac{n-1}{8}>\frac{7n}{8}.
\]
For $k\in G$, set $j=n-k$ and $ u_j=K^{-1}T^kx$. Then $\|u_j\|\le 1$ and
\begin{equation}\label{eq:common-endpoint}
 T^ju_j=K^{-1}T^nx.
\end{equation}
Let $J=\{n-k:k\in G\}$, choose independent Rademacher signs $(\varepsilon_j)_{j\in J}$, and put
\[
 Y=\sum_{j\in J}\varepsilon_ju_j.
\]

Choose $f\in B_{X^*}$ such that $f(T^nx)=\|T^nx\|$. We claim that, for every $j\in J$,
\[
 \mathbb P\bigl(\|T^jY\|\geq K^{-1}\|T^nx\|\bigr)\geq\frac12.
\]
Indeed, fix all signs except $\varepsilon_j$. By \eqref{eq:common-endpoint}, we have 
\[
 f(T^jY)=b+\varepsilon_jK^{-1}\|T^nx\|
\]
for some scalar $b$ independent of $\varepsilon_j$. Since
\[
 \max\{|b+a|,|b-a|\}\geq a
\]
for every scalar $b$ and $a\ge 0$, at least one choice of $\varepsilon_j$ satisfies
$|f(T^jY)|\geq K^{-1}\|T^nx\|$. This proves the claim.

Let
\[
 Z=\card\{j\in J:\|T^jY\|\geq K^{-1}\|T^nx\|\}.
\]
Then $\mathbb EZ\geq\card(J)/2$. Since $0\leq Z\leq\card(J)$, we have 
\[
 \mathbb P\left(Z\geq\frac{\card(J)}4\right)\geq\frac13.
\]
Moreover, $\|u_j\|\le 1$ and $\card(J)\leq n$, so
\[
 \mathbb E\|Y\|\leq\Rad_X(n).
\]
Markov's inequality therefore yields
\[
 \mathbb P\bigl(\|Y\|\leq4\Rad_X(n)\bigr)\geq\frac34.
\]
Hence there is a choice of signs such that
\[
 Z\geq\frac{\card(J)}4
 \quad\text{and}\quad
 \|Y\|\leq4\Rad_X(n).
\]

Suppose that $\|T^nx\|>4K^2\Rad_X(n)$. For every $j$ counted by $Z$, we then have
\[
 \|T^jY\|\geq K^{-1}\|T^nx\|>K\|Y\|.
\]
Moreover,
\[
 Z\geq\frac{\card(J)}4>\frac{7n}{32}>\frac n8.
\]
In particular $Y\neq0$. Thus the unit vector $Y/\|Y\|$ satisfies
\[
 \frac1n\card\left\{1\leq j\leq n:
 \left\|T^j\frac{Y}{\|Y\|}\right\|>K\right\}>\frac18,
\]
contrary to $\Phi_T(K)<2^{-3}$. Hence $\|T^nx\|\leq4K^2\Rad_X(n)$. Since $x\in B_X$ was arbitrary, we have 
\[
 \|T^n\|\leq4K^2\Rad_X(n),\quad \forall \ n\ge 1.
\]

If $X$ has Rademacher type $p$, then there is $C_X>0$ such that
\[
 \Rad_X(n)\leq C_Xn^{1/p},\quad \forall \ n\ge 1.
\]
If $T$ is mean-L-stable, choose $K\ge 1$ with $\Phi_T(K)<2^{-3}$. 
The preceding estimate shows 
\[
 \|T^n\|\leq4K^2C_Xn^{1/p},
\]
and hence $\|T^n\|=O(n^{1/p})$.

\textup{(ii)} 
Let $X$ be a Banach lattice and let $T$ be positive. 
Fix $n\ge 1$ and $x\in B_X\cap X_+$. Put
\[
 G=\{0\leq k<n:\|T^kx\|\leq K\}\text{ and }
 J=\{n-k:k\in G\}.
\]
If $n=1$, then $\card(J)=1>n/2$. If $n\geq2$, the definition of $\Phi_T(K)$ with $N=n-1$ yields 
\[
 \card(J)=\card(G)>\frac n2.
\]
For $j=n-k\in J$, set $u_j=K^{-1}T^kx$. Then $u_j\in B_X\cap X_+$ and
\[
 T^ju_j=K^{-1}T^nx, \quad \forall \ j\in J.
\]
Define $y=\bigvee_{j\in J}u_j$. Then $\|y\|\leq\Env_X(n)$ and, by positivity,
\[
 T^jy\geq T^ju_j=K^{-1}T^nx, \quad \forall \ j\in J.
\]
Suppose that $\|T^nx\|>K^2\Env_X(n)$. Then $y\neq0$, and for every $j\in J$,
\[
 \left\|T^j\frac{y}{\|y\|}\right\|
 \geq\frac{K^{-1}\|T^nx\|}{\|y\|}>K.
\]
Since $\card(J)>n/2$, this contradicts $\Phi_T(K)<2^{-1}$. Therefore
\[
 \|T^nx\|\leq K^2\Env_X(n)
\]
for every $x\in B_X\cap X_+$. For arbitrary $z\in B_X$, positivity gives $|T^nz|\leq T^n|z|$,  and hence $\|T^nz\|\leq K^2\Env_X(n)$. Thus for every $n\ge 1$, 
\[
 \|T^n\|\leq K^2\Env_X(n).
\]

Finally, if $X$ is $p$-convex with constant $M_p(X)$, then for $x_1,\ldots,x_n\in B_X\cap X_+$,
\[
 \|x_1\vee\cdots\vee x_n\|
 \leq
 \biggl\|\biggl(\sum_{j=1}^n x_j^p\biggr)^{1/p}\biggr\|
 \leq M_p(X)n^{1/p}.
\]
Consequently, $\Env_X(n)\leq M_p(X)n^{1/p}$, and therefore
\[
 \|T^n\|\leq K^2 M_p(X)n^{1/p}.
\]
This proves \textup{(ii)}.
\end{proof}

Hilbert spaces have Rademacher type $2$, whereas $\ell^p$ has Rademacher type $\min\{p,2\}$; see, for example, \cite{AK16}.  
Hence every mean-L-stable operator on a Hilbert space satisfies $\|T^n\|=O(n^{1/2})$, while on $\ell^p$ one has the preliminary estimate
\[
 \|T^n\|=O\bigl(n^{1/\min\{p,2\}}\bigr).
\]
For positive operators on $\ell^p$ the order argument improves this to $O(n^{1/p})$.  

For positive operators on an AM-space, the lattice estimate reduces to power boundedness. 
Thus, for positive operators on the classical space $c_0$, mean-L-stability is equivalent to power boundedness.
Here the AM-space order structure of $c_0$ is essential.
It should not be confused with a $c_0$-direct sum of nontrivial Banach lattices.  
The spaces $X_p$ in Section~\ref{sec:critical} have a $c_0$ outer norm, but their internal $\ell^p$ blocks allow lattice suprema of $n$ unit vectors to have norm $n^{1/p}$.

The following proposition follows from Theorem~\ref{thm:relationships}\textup{(ii)} and \cite[Theorem~4.4]{CCEL20}.
We nevertheless give a direct proof based on Proposition~\ref{prop:Hilbert-square}.

\begin{proposition}\label{prop:Hilbert-strict-growth}
Let $T$ be a mean-L-stable operator on a Hilbert space $H$. Then there is $c_T\in(0,1/2)$ such that
\[
 \|T^n\|=O(n^{1/2-c_T}).
\]
\end{proposition}
\begin{proof}
Choose $R\ge 1$ with $\Phi_T(R)<2^{-10}$.  
By Proposition~\ref{prop:Hilbert-square},
\[
 \sum_{k=0}^{N-1}\|T^kx\|^2\le C_2N\|x\|^2
\]
for some $C_2>0$.  
If $T$ is not nilpotent, fix $N\ge1$ and choose $x\in B_H$ with $\|T^Nx\|\ge\|T^N\|/2$.  Since
\[
 \|T^Nx\|\le\|T^{N-k}\|\,\|T^kx\|,
\]
we obtain
\[
 \|T^N\|^2\sum_{r=1}^N\frac1{\|T^r\|^2}\le4C_2N.
\]
Theorem~\ref{thm:uniform-common}\textup{(i)} and Hilbertian type $2$ yield the preliminary estimate $\|T^n\|=O(\sqrt n)$.  
Lemma~\ref{lem:reciprocal-sequence}, applied to $a_n=\|T^n\|^2$, now yields $\|T^n\|=O(n^{1/2-c_T})$ for some $c_T>0$. After decreasing $c_T$ if necessary, we may assume that $c_T<1/2$. The nilpotent case is immediate.
\end{proof}

We next consider positive operators on $p$-convex Banach lattices.  
The quantity $\mathfrak C_{p,X}$ introduced in Section~\ref{sec:prelim} measures the lower $p$-estimate required in the argument below.

\begin{theorem}\label{thm:positive-reciprocal}
Let $T$ be a positive operator on a $p$-convex Banach lattice $X$, where $1\le p<\infty$. 
Suppose that $\Phi_T(R)<2^{-12}$ for some $R\ge 1$.
\begin{enumerate}[label=\textup{(\roman*)}]
    \item If $T$ is not nilpotent, then
\begin{equation}\label{eq:positive-reciprocal-new}
 \|T^N\|^p\sum_{r=1}^m\frac1{\|T^r\|^p}
 \le 2M_p(X)^pR^{2p}N\,\mathfrak C_{p,X}(m)^p
\end{equation}
for all $N,m\in\N$ with $1\le m\le N/8$. 
\item If $\mathfrak C_{p,X}(m)^p=o(\log(m+1))$,
then 
\[
\|T^n\|=o(n^{1/p}).
\]
\item If $\sup_{m\ge 1}\mathfrak C_{p,X}(m)<\infty$, then there are $C\ge 1$ and $c_T\in(0,1/p)$ such that
\[
\|T^n\|\le Cn^{1/p-c_T}, \quad \forall \ n\ge 1.
\]
\end{enumerate}
\end{theorem}

\begin{proof}
We first prove \textup{(i)}. 
Since $\Phi_T(R)<2^{-12}$,
\[
 \frac1N\card\{1\leq j\leq N:\|T^jx\|>R\}<2^{-12}
\]
for every $x\in B_X$ and every $N\ge 1$. 
Since $T$ is not nilpotent, $\|T^r\|>0$ for every $r\ge 1$.

Fix $N,m\in\N$ with $1\le m\le N/8$, and let $x\in B_X\cap X_+$  such that $T^Nx\ne 0$.
Put
\[
 E=\|T^Nx\|,
 \quad
 H_N=\left\lfloor\frac N4\right\rfloor,
 \quad
 G=\{0\le k<N:\|T^kx\|\le R\}, \quad S_m=\sum_{r=1}^m\frac1{\|T^r\|^p}.
\]
Since $m\ge 1$, we have $N\ge8$, and hence $H_N\ge 1$. Moreover, $0\in G$. For $1\le j\le H_N$, put
\[
 q_j=\sum_{r=1}^m
 \frac{\mathbf1_{\{N-j-r\notin G\}}}{\|T^r\|^p}.
\]
For the indices under consideration, $0<N-j-r<N$. Applying the density estimate to $x$ with $N-1$ in place of $N$, we obtain
\[
 \card\{0\le k<N:k\notin G\}<2^{-12}(N-1)<2^{-12}N.
\]
For each fixed $r$, the map $j\mapsto N-j-r$ is injective. Therefore
\[
 \sum_{j=1}^{H_N}q_j\le2^{-12}NS_m,
\]
so fewer than $2^{-11}N$ indices $j$ satisfy $q_j>S_m/2$.

Set
\[
 w=\biggl(\sum_{k\in G}(T^kx)^p\biggr)^{1/p}.
\]
Since $0\in G$, we have $w\ne 0$. By $p$-convexity, $\|w\|\le M_p(X)RN^{1/p}$. Applying the density estimate to $w/\|w\|$ with $N=H_N$, fewer than $2^{-12}H_N$ indices $1\le j\le H_N$ satisfy
\[
 \|T^jw\|>R\|w\|.
\]
Since $H_N\ge N/8$, we have
\[
 2^{-11}N+2^{-12}H_N
 \le \left(2^{-8}+2^{-12}\right)H_N<H_N.
\]
Hence there is $1\le j\le H_N$ such that
\[
 q_j\le\frac{S_m}{2}
 \quad\text{and}\quad
 \|T^jw\|\le R\|w\|.
\]

Let $I=\{1\le r\le m:N-j-r\in G\}$. Then
\[
 \sum_{r\in I}\frac1{\|T^r\|^p}\ge\frac{S_m}{2}.
\]
For every $r\in I$, we have 
\[
 E=\|T^rT^{N-r}x\|\le\|T^r\|\,\|T^{N-r}x\|,
\]
and therefore
\[
 \sum_{r\in I}\|T^{N-r}x\|^p\ge\frac12E^pS_m.
\]
Put
\[
 Y=\biggl(\sum_{r\in I}(T^{N-r}x)^p\biggr)^{1/p}.
\]
By the definition of $\mathfrak C_{p,X}(m)$,
\[
 E\left(\frac{S_m}{2}\right)^{1/p}
 \le \mathfrak C_{p,X}(m)\|Y\|.
\]

We use the standard lattice Minkowski inequality
\[
 \biggl(\sum_i(Sy_i)^p\biggr)^{1/p}
 \le S\biggl(\biggl(\sum_i y_i^p\biggr)^{1/p}\biggr)
\]
for a positive operator $S$ and positive vectors $y_i$. For $p>1$, this follows from the finite lattice functional calculus by applying positivity to
\[
\sum_i\alpha_i y_i\le\biggl(\sum_i y_i^p\biggr)^{1/p}
\]
whenever $\alpha_i\ge0$ and $\|(\alpha_i)\|_q\le1$, where $q$ is conjugate to $p$.  For $p=1$ it is an equality. Since $N-j-r\in G$ for $r\in I$, it follows that
\[
 Y\le T^j\biggl(\biggl(\sum_{r\in I}(T^{N-j-r}x)^p\biggr)^{1/p}\biggr)
 \le T^jw.
\]
Consequently, we have 
\[
 \|Y\|\le R\|w\|\le M_p(X)R^2N^{1/p}.
\]
Combining the last two estimates and taking the $p$th power of both sides, we obtain 
\[
 E^pS_m
 \le 2M_p(X)^pR^{2p}N\,\mathfrak C_{p,X}(m)^p.
\]
The same estimate is trivial when $T^Nx=0$. 
Since $T^N$ is positive,
\[
 \|T^N\|=\sup_{x\in B_X\cap X_+}\|T^Nx\|.
\]
Taking the supremum over $x\in B_X\cap X_+$ proves \eqref{eq:positive-reciprocal-new}.

We next prove \textup{(ii)}. 
The conclusion is immediate if $T$ is nilpotent, so assume that $T$ is not nilpotent.
Put $a_n=\|T^n\|^p$. 
By Theorem~\ref{thm:uniform-common}\textup{(ii)}, there is $A\ge 1$ such that $a_n\le An$ for every $n\ge 1$. 
Taking $m=\lfloor N/8\rfloor$ in \eqref{eq:positive-reciprocal-new}, we obtain
\[
 \sum_{r=1}^m\frac1{a_r}
 \ge\frac1A\sum_{r=1}^m\frac1r
 \ge c\log(N+1)
\]
for $N\ge8$. Hence
\[
 \frac{a_N}{N}
 \le C\frac{\mathfrak C_{p,X}(\lfloor N/8\rfloor)^p}{\log(N+1)}.
\]
Thus $\mathfrak C_{p,X}(m)^p=o(\log(m+1))$ implies $a_N/N\to0$, and hence $\|T^N\|=o(N^{1/p})$.

Finally, suppose that $\sup_{m\ge 1} \mathfrak C_{p,X}(m)<\infty$.
Again, there is nothing to prove if $T$ is nilpotent.
Otherwise, \eqref{eq:positive-reciprocal-new} with $m=\lfloor N/8\rfloor$ satisfies the hypothesis of Lemma~\ref{lem:reciprocal-sequence}. 
Hence $a_n\le Cn^{1-\delta}$ for some $\delta>0$.
After replacing $\delta$ by a smaller number if necessary, we may assume $0<\delta<1$. 
Taking $p$th roots and putting $c_T=\delta/p$ yields
\[
 \|T^n\|\le Cn^{1/p-c_T},
\]
where $0<c_T<1/p$.
\end{proof}

Now we are ready to prove Theorem~\ref{thm:main-growth}.

\begin{proof}[Proof of Theorem~\ref{thm:main-growth}]
Part \textup{(i)} follows from Theorem~\ref{thm:uniform-common}\textup{(i)}. Since every Banach space has Rademacher type $1$, the estimate $\|T^n\|=O(n)$ holds without any additional geometric assumption.

Part \textup{(ii)} is Proposition~\ref{prop:Hilbert-strict-growth}.

For \textup{(iii)}, the $p$-convex estimate follows from Theorem~\ref{thm:uniform-common}\textup{(ii)}. If $X$ is an abstract $L^p$-space, then $\mathfrak C_{p,X}(m)=1$, and the strict estimate follows from Theorem~\ref{thm:positive-reciprocal}.
\end{proof}

\section{Weighted shifts on dyadic \texorpdfstring{$c_0$}{c0}-sums}\label{sec:critical}

We now prove Theorem~\ref{thm:critical-family-main}.  
Polynomial weighted shifts with weights $((m+1)/m)^\gamma$, $0<\gamma<1/p$, are absolutely Ces\`aro bounded on $\ell^p$, and the corresponding examples are topologically mixing \cite[Theorem~2.1 and Corollary~2.3]{BBMP20}.  
We take the critical exponent $\gamma=1/p$ and replace $\ell^p$ by the dyadic $c_0$-sum $X_p$. The outer $c_0$-norm separates the dyadic blocks and is crucial for the counting argument below.

Throughout this section, all spaces are complex.  
For $1\le p<\infty$, set
\[
 I_j=\{2^j,\ldots,2^{j+1}-1\},
 \quad
 X_p=\biggl(\bigoplus_{j\ge 0}\ell^p(I_j)\biggr)_{c_0}.
\]
Thus
\[
 \|x\|_{X_p}
 =\sup_{j\ge 0}\biggl(\sum_{m\in I_j}|x_m|^p\biggr)^{1/p},
\]
and the block norms tend to zero.  
Let $T_p$ be the unilateral weighted backward shift on $X_p$ with weights $\bigl(\frac{m+1}{m}\bigr)^{1/p}$, that is 
\[
 (T_px)_m=\left(\frac{m+1}{m}\right)^{1/p}x_{m+1}.
\]
Then
\[
 (T_p^nx)_m=\left(\frac{m+n}{m}\right)^{1/p}x_{m+n}.
\]
If $I_j=[a,2a)$, then
\[
 \|(T_px)|_{I_j}\|_p^p
 \le2\bigl(\|x|_{I_j}\|_p^p+\|x|_{I_{j+1}}\|_p^p\bigr).
\]
Hence $T_p$ maps $X_p$ into itself and $\|T_p\|\le4^{1/p}$.

Recall that an operator $T$ on a Banach space $X$ is \emph{topologically mixing} if for any nonempty open subsets $U$ and $V$ of $X$ there exists $N\ge 1$ such that $T^n(U)\cap V\neq\emptyset$ for all $n\geq N$.

\begin{proof}[Proof of Theorem~\ref{thm:critical-family-main}]
The space $X_p$ is $p$-convex with constant one.  For a finite family $x_1,\ldots,x_m\in X_p$,
\[
 \Biggl\|\biggl(\sum_{i=1}^m|x_i|^p\biggr)^{1/p}\Biggr\|_{X_p}^p
 =\sup_{j\ge 0}\sum_{i=1}^m\|x_i|_{I_j}\|_p^p
 \le\sum_{i=1}^m\|x_i\|_{X_p}^p.
\]
First take $p=1$ and write $X=X_1$ and $T=T_1$.  
The general case will then follow from a simple identity.  
The power formula above shows
\[
 T^ne_{n+1}=(n+1)e_1,
\]
so $\|T^n\|\ge n+1$.  
For the reverse estimate, fix an output block $I_j=[a,2a)$, where $a=2^j$.  
Then
\[
 \sum_{m=a}^{2a-1}|(T^nx)_m| \le\left(1+\frac na\right) \sum_{r=a+n}^{2a-1+n}|x_r|.
\]
The interval $[a+n,2a+n)$ has length $a$ and left endpoint at least $a$, hence it meets at most two dyadic blocks.  Therefore
\[
 \sum_{r=a+n}^{2a-1+n}|x_r|\le2\|x\|_X,
\]
and $1+n/a\le n+1$.  Taking the supremum over $j$ yields
\[
 n+1\le\|T^n\|\le2(n+1).
\]

We next prove the weak orbit estimate.  We use the one-sided discrete maximal function
\[
 M^+u(n)=\sup_{h\ge 1}\frac1h\sum_{r=n+1}^{n+h}u_r
\]
for finitely supported $u\ge 0$.  
After extending $u$ by zero to $\mathbb Z$, $M^+$ is pointwise dominated, up to an absolute constant, by the discrete Hardy--Littlewood maximal operator. The standard weak $(1,1)$ estimate for the latter \cite[Section~2.4]{Bakas21} therefore yields 
\[
\card\{n\in\N:M^+u(n)>\alpha\}
\le\frac{C}{\alpha}\sum_{r\ge 1}u_r,
\]
where $C>0$ is an absolute constant. 
Assume $\|x\|_X=1$. For $0<\lambda\le4$, the estimate follows after enlarging the absolute constant, so suppose that $\lambda>4$. 
If $\|T^nx\|>\lambda$, then for some dyadic $a$,
\[
 \sum_{m=a}^{2a-1}\frac{m+n}{m}|x_{m+n}|>\lambda.
\]
If $a\ge n$, the left-hand side is at most $4$, a contradiction. Thus $a<n$. 
For $a\le m<2a$ one has $m+n<3n$, and the preceding inequality implies
\[
 \sum_{r=n+a}^{n+2a-1}|x_r|>\frac{a\lambda}{3n}.
\]
Taking $h=2a$ in the definition of $M^+$, we obtain
\[
 M^+(|x|)(n)>\frac{\lambda}{6n}.
\]
If $2^q\le n<2^{q+1}$, then the interval $[n+1,n+2a]$ is contained in $[2^q,2^{q+3})$. 
Restrict $|x|$ to this three-block interval.  
Its $\ell^1$ norm is at most $3$, while the preceding maximal-function estimate implies a maximal average larger than $\lambda/(12\cdot 2^q)$.  
By the weak $(1,1)$ estimate for $M^+$, there is an absolute constant $C>0$ such that
\[
 \card\{2^q\le n<2^{q+1}:\|T^nx\|>\lambda\}
 \le C\frac{2^q}{\lambda}.
\]
Summing over the dyadic time blocks meeting $[1,N]$ yields
\[
 \frac1N\card\{1\le n\le N:\|T^nx\|>\lambda\}
 \le\frac{C}{\lambda}.
\]
Thus there is an absolute constant $C_0>0$ such that, by homogeneity,
\[
 \frac1N\card\{1\le n\le N:\|T^nx\|>\lambda\}
 \le C_0\frac{\|x\|}{\lambda}.
\]
In terms of the uniform density modulus introduced in Section~\ref{sec:prelim}, this estimate gives
\[
 \mathcal K_{T}(\eta)
 \le
 \max\left\{1,\frac{C_0}{\eta}\right\},
 \quad \forall \ 0<\eta<1.
\]
In particular, Proposition~\ref{prop:density-characterization} shows that $T$ is mean-L-stable.

We now pass to general $p$.  
For $p\ge 1$, define
\[
 \Gamma_p(x)=(|x_m|^p)_m.
\]
Then
\begin{equation}\label{eq:convexification-new}
 \|\Gamma_p(x)\|_{X_1}=\|x\|_{X_p}^p\text{ and } \Gamma_p(T_p^nx)=T_1^n\Gamma_p(x).
\end{equation}
Since every positive element of $X_1$ is $\Gamma_p(x)$ for some $x\in X_p$, positivity also yields
\[
 \|T_p^n\|^p=\|T_1^n\|.
\]
The $p=1$ estimate and this identity prove the asserted bound for $\|T_p^n\|$.  The weak estimate for $T_1$, together with \eqref{eq:convexification-new}, implies
\[
 \frac1N\card\{1\le n\le N:\|T_p^nx\|>\lambda\}
 \le C_0\left(\frac{\|x\|}{\lambda}\right)^p.
\]  
It also shows that $T_p$ is mean-L-stable. 

To prove topological mixing, let $X_{p,0}$ be the dense subspace of finitely supported vectors in $X_p$ and, for $y\in X_{p,0}$, define
\[
 (R_{p,n}y)_{m+n}=\left(\frac{m}{m+n}\right)^{1/p}y_m,
 \quad
 (R_{p,n}y)_r=0\text{ otherwise}.
\]
Then $T_p^nR_{p,n}y=y$, $R_{p,n}y\to0$, and $T_p^ny\to0$.  
To verify mixing directly, let $U,V\subset X_p$ be nonempty open sets and choose finitely supported $u\in U$ and $v\in V$.  For all sufficiently large $n$, $T_p^nu=0$ and $u+R_{p,n}v\in U$, while
\[
 T_p^n(u+R_{p,n}v)=v\in V.
\]
Thus $T_p^n(U)\cap V\ne\varnothing$ for every sufficiently large $n$.

It remains to estimate the averages of the orbit norms.  Normalize $\|x\|\le 1$.  The weak estimate and the power estimate already proved imply, for $1\le n\le N$,
\[
 \|T_p^nx\|\le C_pN^{1/p},
\]
and
\[
 \frac1N\card\{1\le n\le N:\|T_p^nx\|>\lambda\}
 \le\min\{1,C_0\lambda^{-p}\}.
\]
Using the layer-cake identity and the weak $(p,p)$ estimate, we obtain
\[
\begin{aligned}
\frac1N\sum_{n=1}^N\|T_p^nx\|^s
&=s\int_0^{C_pN^{1/p}}\lambda^{s-1}
\frac1N\card\{1\le n\le N:\|T_p^nx\|>\lambda\}\,\mathrm d\lambda\\
&\le  s\int_0^1\lambda^{s-1}\,\mathrm d\lambda
+sC_0\int_1^{C_pN^{1/p}}\lambda^{s-p-1}\,\mathrm d\lambda.
\end{aligned}
\]
The last expression is uniformly bounded in $N$ when $s<p$, is $O(\log(N+1))$ when $s=p$, and is $O(N^{s/p-1})$ when $s>p$. This proves the required upper bounds.

For the lower bounds, take $x=e_{N+1}$.  Then for $1\le n\le N$,
\[
 \|T_p^ne_{N+1}\|
 =\left(\frac{N+1}{N+1-n}\right)^{1/p}.
\]
Therefore
\[
 \frac1N\sum_{n=1}^N\|T_p^ne_{N+1}\|^s
 =\frac{(N+1)^{s/p}}{N}\sum_{k=1}^Nk^{-s/p}.
\]
The three cases $s<p$, $s=p$, and $s>p$ now follow from the standard estimates for this power sum.  
This proves the asserted estimates for the averages. Taking $s=1$ and adding the term $n=0$ shows that $T_p$ is absolutely Ces\`aro bounded for $p>1$, whereas for $p=1$ the supremum of the corresponding Ces\`aro averages grows like $\log(N+1)$.
\end{proof}

For completeness, we compute the two lattice quantities used above for $\ell^p$ and $X_p$. These computations also clarify why the general $p$-convex estimate is sharp on $X_p$.

For positive $x_1,\ldots,x_n\in B_{X_p}$,
\[
 \|x_1\vee\cdots\vee x_n\|_{X_p}^p
 \le\biggl\|\sum_{i=1}^n x_i^p\biggr\|_{X_1}\le n,
\]
so $\Env_{X_p}(n)\le n^{1/p}$.  Distinct coordinate vectors inside one sufficiently large dyadic block attain equality.  Thus $\Env_{X_p}(n)=n^{1/p}$.

For $\ell^p$,
\[
 \Biggl\|\biggl(\sum_{i=1}^m|y_i|^p\biggr)^{1/p}\Biggr\|_p^p
 =\sum_{i=1}^m\|y_i\|_p^p,
\]
so $\mathfrak C_{p,\ell^p}(m)=1$.  
On $X_p$, monotonicity yields the upper bound $m^{1/p}$.
For the reverse inequality, choose one coordinate vector in each of $m$ distinct dyadic blocks.  
The individual norms are one, whereas the lattice $p$-sum has $X_p$-norm one.  
Hence $\mathfrak C_{p,X_p}(m)=m^{1/p}$.

Thus $\Env_{X_p}(n)=n^{1/p}$ and $\mathfrak C_{p,X_p}(m)=m^{1/p}$.  
The estimate from Theorem~\ref{thm:uniform-common}\textup{(ii)} therefore has the same order as $\|T_p^n\|$.  
Also, the reciprocal estimate in Theorem~\ref{thm:positive-reciprocal} does not lower this power.  
For the classical space $c_0$, one has $\Env_{c_0}(n)=1$ because $c_0$ is an AM-space.

\section{Mean Li--Yorke chaos and distributional chaos}\label{sec:chaos}

We conclude with two applications to linear chaos. 
Throughout this section, distributional chaos means distributional chaos of type $1$. 
Let $T$ be an operator on a Banach space $X$. 
A pair $(x,y)\in X\times X$ is called a \emph{mean Li--Yorke pair} if
\[
\liminf_{N\to\infty}\frac1N\sum_{n=1}^N\|T^nx-T^ny\|=0
\text{ and }
\limsup_{N\to\infty}\frac1N\sum_{n=1}^N\|T^nx-T^ny\|>0.
\]
The operator $T$ is \emph{mean Li--Yorke chaotic} if there is an uncountable set $S\subset X$ such that every pair of distinct points in $S$ is a mean Li--Yorke pair. It is \emph{densely mean Li--Yorke chaotic} if such an $S$ can be chosen dense in $X$.

A pair $(x,y)$ is \emph{distributionally chaotic} if
\[
\dens\{n\in\N:\|T^nx-T^ny\|<\varepsilon\}=1
\]
for every $\varepsilon>0$, and there is $\delta>0$ such that
\[
\dens\{n\in\N:\|T^nx-T^ny\|>\delta\}=1.
\]
The operator is called distributionally chaotic if it admits an uncountable scrambled set of such pairs, and densely distributionally chaotic if the scrambled set can be chosen dense.

For continuous maps on compact metric spaces, mean Li--Yorke chaos is equivalent to distributional chaos of type $2$ \cite{D14}, while distributional chaos of type $1$ is stronger \cite{SS04}. 
For operators on Banach spaces, the notions of distributional chaos of types $1$ and $2$ are equivalent \cite[Theorem~2]{BBPW18}, but this property is not equivalent to mean Li--Yorke chaos. 
In particular, \cite[Theorem~25]{BBPW18} constructs distributionally chaotic operators which are not mean Li--Yorke chaotic; see also \cite[Example~13]{BBP20}. 
Bernardes, Bonilla and Peris asked whether the opposite separation can occur, namely, whether there exists a mean Li--Yorke chaotic operator that is not distributionally chaotic \cite[Question~16]{BBP20}.

For later use, we put
\[
\mathcal M_0(T)=\biggl\{x\in X\colon \liminf_{N\to\infty}\frac1N\sum_{n=1}^N\|T^nx\|=0\biggr\}.
\]
If $X$ is separable, Theorems~4 and~17 of \cite{BBP20} imply that $T$ is densely mean Li--Yorke chaotic if and only if $T$ is not absolutely Ces\`aro bounded and $\mathcal M_0(T)$ is dense in $X$. Indeed, Theorem~4 of \cite{BBP20} shows that failure of absolute Ces\`aro boundedness yields a residual set of vectors with unbounded Ces\`aro orbit averages, while $\mathcal M_0(T)$ is a $G_\delta$ set. Hence, if $\mathcal M_0(T)$ is dense, then it is residual.

Recall that a pair $(x,y)$ is \emph{distributionally asymptotic} if 
\[
\dens\{n\in\N\colon \|T^nx-T^ny\|>\varepsilon\}=0
\]
for every $\varepsilon>0$.

\begin{proposition}\label{prop:T1-chaos}
Let $T_1$ be the operator on $X_1$ in Theorem~\ref{thm:critical-family-main}. 
Then $T_1$ is densely mean Li--Yorke chaotic but not distributionally chaotic. 
More precisely, every pair in $X_1\times X_1$ is distributionally asymptotic.
\end{proposition}

\begin{proof}
By Theorem~\ref{thm:critical-family-main}, $T_1$ is not absolutely Ces\`aro bounded. The space $X_1$ is separable, and its finitely supported vectors are dense. Every such vector belongs to $\mathcal M_0(T_1)$, since its orbit is eventually zero. Hence $\mathcal M_0(T_1)$ is dense in $X_1$, and the preceding characterization shows that $T_1$ is densely mean Li--Yorke chaotic.

It remains to prove the stronger assertion that every pair is distributionally asymptotic. 
Fix $x\in X_1$ and $\varepsilon>0$. Let $z$ be finitely supported. Since $T_1^nz=0$ for all sufficiently large $n$, the sets
\[
\{n\in\N:\|T_1^nx\|>\varepsilon\}
\quad\text{and}\quad
\{n\in\N:\|T_1^n(x-z)\|>\varepsilon\}
\]
differ by at most finitely many integers. The weak orbit estimate in Theorem~\ref{thm:critical-family-main} therefore yields
\[
\dens\{n\in\N:\|T_1^nx\|>\varepsilon\}
\le C_0\frac{\|x-z\|}{\varepsilon}.
\]
Letting $z\to x$ through finitely supported vectors, we obtain
\[
\dens\{n\in\N:\|T_1^nx\|>\varepsilon\}=0.
\]
Applying this conclusion to $x-y$ shows that, for every $x,y\in X_1$ and every $\varepsilon>0$, 
\[
\dens\{n\in\N:\|T_1^nx-T_1^ny\|>\varepsilon\}=0.
\]
Thus every pair is distributionally asymptotic. 
In particular, $T_1$ has no distributionally chaotic pair.
\end{proof}

Proposition~\ref{prop:T1-chaos} answers Question~16 of \cite{BBP20}; see also Question~\ref{ques:mean-chaos} in the present paper. 
The example has the stronger property that all pairs are distributionally asymptotic.

We also obtain the converse implication under the geometric and order assumptions for which mean-L-stability and absolute Ces\`aro boundedness coincide. 
Recall that a pair $(x,y)$ is \emph{distributionally proximal} if
\[
\dens\{n\in\N:\|T^nx-T^ny\|<\varepsilon\}=1
\]
for every $\varepsilon>0$.

\begin{proposition}\label{prop:dense-MLY-DC}
Let $T$ be a densely mean Li--Yorke chaotic operator on a separable Banach space $X$. Assume that one of the following holds.
\begin{enumerate}[label=\textup{(\roman*)}]
\item $X$ is a Hilbert space.
\item $X$ is an AM-space and $T$ is positive.
\item $X$ is an AL-space and $T$ is positive.
\end{enumerate}
Then $T$ is densely distributionally chaotic.
\end{proposition}

\begin{proof}
Since $T$ is densely mean Li--Yorke chaotic, the characterization above shows that $T$ is not absolutely Ces\`aro bounded and that $\mathcal M_0(T)$ is dense in $X$. 
Theorem~\ref{thm:relationships} yields that absolute Ces\`aro boundedness and mean-L-stability are equivalent in all three cases. 
It follows that $T$ is not mean-L-stable.

We claim that the distributionally proximal relation of $T$ is dense in $X\times X$. 
Let $x\in\mathcal M_0(T)$. 
There is an increasing sequence $(N_k)$ such that
\[
\frac1{N_k}\sum_{n=1}^{N_k}\|T^nx\|\longrightarrow0.
\]
For every $\varepsilon>0$, Markov's inequality gives
\[
\frac1{N_k}\card\{1\le n\le N_k:\|T^nx\|\ge\varepsilon\}
\le\frac1{\varepsilon N_k}\sum_{n=1}^{N_k}\|T^nx\|\longrightarrow0.
\]
Hence
\[
\dens\{n\in\N:\|T^nx\|<\varepsilon\}=1
\]
for every $\varepsilon>0$, so $(x,0)$ is distributionally proximal.

Now let $U,V\subset X$ be nonempty open sets. Since $U-V$ is a nonempty open set and $\mathcal M_0(T)$ is dense, choose
\[
z\in\mathcal M_0(T)\cap(U-V).
\]
Write $z=u-v$ with $u\in U$ and $v\in V$. Since $(z,0)$ is distributionally proximal, so is $(u,v)$. 
Thus the distributionally proximal relation is dense in $X\times X$.

 By the mean-L-stability dichotomy and the characterization of dense distributional chaos in \cite[Theorems~5.13, 5.41 and~5.42]{JL25}, 
 the fact that $T$ is not mean-L-stable and has a dense distributionally proximal relation implies that $T$ is densely distributionally chaotic.
\end{proof}

For Hilbert spaces, the density assumption can be removed.

\begin{corollary}\label{cor:Hilbert-MLY-DC}
Every mean Li--Yorke chaotic operator on a Hilbert space is distributionally chaotic.
\end{corollary}

\begin{proof}
Let $T$ be mean Li--Yorke chaotic on a Hilbert space $H$. By the proof of \cite[Theorem~19]{BBP20}, there is an infinite-dimensional separable closed $T$-invariant subspace $Y$ of $H$ such that $T|_Y$ is densely mean Li--Yorke chaotic. Since $Y$ is a Hilbert space, Proposition~\ref{prop:dense-MLY-DC} shows that $T|_Y$ is densely distributionally chaotic. Hence $T$ is distributionally chaotic.
\end{proof}

\subsection*{Acknowledgments} 
This research was supported by the National Natural Science Foundation of China (Grant Nos. 12222110, 12031019).


\begin{thebibliography}{99}

\bibitem{AK16}
F. Albiac and N. J. Kalton,
\emph{Topics in Banach Space Theory},
2nd ed., Graduate Texts in Mathematics, vol.~233, Springer, Cham, 2016.

\bibitem{AS16}
A. Aleman and L. Suciu,
\emph{On ergodic operator means in Banach spaces},
Integral Equations Operator Theory \textbf{85} (2016), no.~2, 259--287.

\bibitem{AC26}
L. Arnold and C. Cuny,
\emph{On the growth rate of powers of a strongly Kreiss bounded operator on an $L^p$-space},
Studia Math. \textbf{288} (2026), no.~1, 1--25.

\bibitem{Bakas21}
O.~Bakas, 
\emph{On a Problem of Pichorides},
J. Geom. Anal. \textbf{31} (2021), 7455--7512.

\bibitem{BBMP20}
T. Berm\'udez, A. Bonilla, V. M\"uller and A. Peris,
\emph{Ces\`aro bounded operators in Banach spaces},
J. Anal. Math. \textbf{140} (2020), no.~1, 187--206.

\bibitem{BBMP13}
N. C. Bernardes Jr.,  A. Bonilla,  V. M\"uller and A. Peris, \emph{Distributional chaos for linear operators}, 
J. Funct. Anal. \textbf{265} (2013), no. 9, 2143--2163.

\bibitem{BBP20}
N. C. Bernardes Jr., A. Bonilla and A. Peris,
\emph{Mean Li--Yorke chaos in Banach spaces},
J. Funct. Anal. \textbf{278} (2020), no.~3, Paper No.~108343.

\bibitem{BBPW18}
N. C. Bernardes Jr., A. Bonilla,  A. Peris and X. Wu,  \emph{Distributional chaos for operators on Banach spaces}, J. Math. Anal. Appl. \textbf{459} (2018), no. 2, 797--821. 

\bibitem{BM21}
A. Bonilla and V. M\"uller,
\emph{Kreiss bounded and uniformly Kreiss bounded operators},
Rev. Mat. Complut. \textbf{34} (2021), no.~2, 469--487.

\bibitem{CCEL20}
G. Cohen, C. Cuny, T. Eisner and M. Lin,
\emph{Resolvent conditions and growth of powers of operators},
J. Math. Anal. Appl. \textbf{487} (2020), no.~2, Paper No.~124035.

\bibitem{C20}
C. Cuny,
\emph{Resolvent conditions and growth of powers of operators on $L^p$ spaces},
Pure Appl. Funct. Anal. \textbf{5} (2020), no.~5, 1025--1038.

\bibitem{DLV24}
C. Deng, E. Lorist and M. Veraar,
\emph{Strongly Kreiss bounded operators in UMD Banach spaces},
Semigroup Forum \textbf{108} (2024), no.~3, 594--625.

\bibitem{D14}
T. Downarowicz, \emph{Positive topological entropy implies chaos DC2}, Proc. Amer. Math. Soc. \textbf{142} (2014), no. 1, 137--149.

\bibitem{EFR02}
O. El-Fallah and T. Ransford,
\emph{Extremal growth of powers of operators satisfying resolvent conditions of Kreiss--Ritt type},
J. Funct. Anal. \textbf{196} (2002), no.~1, 135--154.

\bibitem{F51}
S. Fomin,
\emph{On dynamical systems with a purely point spectrum},
Dokl. Akad. Nauk SSSR (N.S.) \textbf{77} (1951), 29--32 (in Russian).

\bibitem{JL25} 
Z. Jiang and J. Li,  \emph{Chaos for endomorphisms of completely metrizable groups and linear operators on Fréchet spaces}, J. Math. Anal. Appl. \textbf{543} (2025), no. 2, Paper No. 129033. 

\bibitem{K85} 
J.-P. Kahane,  \emph{Some random series of functions}, Second edition, Cambridge Studies in Advanced Mathematics, vol.~5. Cambridge University Press, Cambridge, 1985.

\bibitem{L24}
J. Li, \emph{Trichotomy for the orbits of a hypercyclic operator on a Banach space}, 
Proc. Amer. Math. Soc. \textbf{152} (2024), no. 12, 5207--5217.

\bibitem{LTY15}
J. Li, S. Tu and X. Ye,
\emph{Mean equicontinuity and mean sensitivity},
Ergodic Theory Dynam. Systems \textbf{35} (2015), no.~8, 2587--2612.

\bibitem{LWZ26}
J. Li, X. Wang and J. Zhao,
\emph{Density properties of orbits for a hypercyclic operator on a Banach space},
Canad. Math. Bull. \textbf{69} (2026), no.~1, 61--76.

\bibitem{LY16}
J. Li and X. Ye,
\emph{Recent development of chaos theory in topological dynamics}, Acta Math. Sin. (Engl. Ser.) \textbf{32} (2016), no. 1, 83--114. 

\bibitem{LH15} 
L. Luo and B. Hou, 
\emph{Some remarks on distributional chaos for bounded linear operators}, Turkish J. Math. \textbf{39} (2015), no. 2, 251--258.

\bibitem{M91}
P. Meyer-Nieberg,
\emph{Banach Lattices},
Universitext, Springer-Verlag, Berlin, 1991.

\bibitem{MSZ05}
A. Montes-Rodr\'iguez, J. S\'anchez-\'Alvarez and J. Zem\'anek,
\emph{Uniform Abel--Kreiss boundedness and the extremal behaviour of the Volterra operator},
Proc. Lond. Math. Soc. (3) \textbf{91} (2005), no.~3, 761--788.

\bibitem{SS04} 
J. Sm\'ital and M. \v{S}tef\'ankov\'a, \emph{Distributional chaos for triangular maps},
 Chaos Solitons Fractals \textbf{21} (2004), no. 5, 1125--1128.

\bibitem{V83}
J. A. van Casteren,
\emph{Operators similar to unitary or selfadjoint ones},
Pacific J. Math. \textbf{104} (1983), no.~1, 241--255.

\bibitem{V97}
J. A. van Casteren,
\emph{Boundedness properties of resolvents and semigroups of operators},
in: Linear Operators (Warsaw, 1994), Banach Center Publ. \textbf{38},
Polish Acad. Sci. Inst. Math., Warsaw, 1997, 59--74.

\end{thebibliography}
\end{document}